\documentclass{article}

\usepackage{amssymb,amsfonts,amsmath}
\usepackage[english]{babel}
\usepackage[utf8]{inputenc}
\usepackage{a4wide}
\usepackage{graphicx}
\usepackage{dsfont}
\usepackage{mathrsfs} 
\usepackage{mathtools} 
\usepackage{amsthm}
\usepackage{color}

\usepackage{url}
\newcommand{\PP}{\mathbb P}
\newcommand{\NN}{\mathbb N}

\newcommand{\QQ}{\mathbb Q}

\newcommand{\ZZ}{\mathbb Z}
\newcommand{\FF}{\mathbb F}
\newcommand{\diag}{\mathrm{ diag}}

\newtheorem*{theoremA}{Theorem A}
\newtheorem{theorem}{Theorem}
\newtheorem{coro}{Corollary}
\newtheorem{proposition}{Proposition}

\newtheorem{remark}{Remark}
\newtheorem{lemma}{Lemma}

\begin{document}

\title{A note on matrices over $\ZZ$ with entries stemming from binomial coefficients and from Catalan numbers once pure and once taken modulo $2$}
\author{Roswitha Hofer\thanks{Institute of Financial Mathematics and Applied Number Theory, Johannes Kepler University Linz, Altenbergerstr. 69, 4040 Linz, Austria. e-mail: roswitha.hofer@jku.at}}

\maketitle
\begin{abstract}
The Pascal matrix, which is related to Pascal's triangle, appears in many places in the theory of uniform distribution and in many other areas of mathematics. Examples are the construction of low-discrepancy sequences as well as normal numbers or the binomial transforms of Hankel matrices. Hankel matrices which are defined by Catalan numbers and related to the paperfolding sequence are interesting objects in number theory. Therefore, matrices that share many properties with the Pascal matrix or such Hankel matrices are of interest. In this note we will collect common features of the Pascal matrix and the same modulo $2$ as well as the Hankel matrix defined by Catalan numbers once pure and once modulo $2$ in the ring of integers. Hankel matrices with only $0$ and $1$ entries in e.g. finite fields gave recently access to counterexamples to the so-called $X$-adic Liouville conjecture. This justifies as well as motivates our consideration of further matrices with $0$ and $1$ entries. 
\end{abstract}

\noindent{\textbf{Keywords:} Pascal matrices, Hankel matrices,  Catalan numbers, paperfolding sequences} 
\\
\noindent{\textbf{MSC2010:} 11C20, 11B50.}

\section{Introduction}
In this note we collect and work out nice coincidences between the pure Pascal matrix---that is built by the binomial coefficients---and between this Pascal matrix taken modulo $2$, both considered as matrices in the ring of integers. Further interesting coincidences are detected between different Hankel matrices determined by the Catalan numbers once pure and once modulo $2$. We will organize our note as follows. 
Section~\ref{sec:2} is devoted to the Pascal matrix which is used for constructing digital sequences in the sense of Niederreiter and for constructing normal numbers both satisfying low discrepancy bounds. The noted coincidences with the Pascal matrix modulo $2$, therefore makes this matrix to a potential candidate for constructing further low-discrepancy sequences and further interesting normal numbers. 
Section~\ref{sec:3} is centred on Hankel matrices determined by the Catalan numbers once pure and once modulo $2$. Such Hankel matrices are related to low-discrepancy Kronecker-type sequences. 
Matrices and sequences with $0$-$1$ entries exclusively but considered in a field or ring with more than two elements, might be interesting in various branches in mathematics. One recent example is the Hankel matrix determined via a paperfolding $0$-$1$ sequence $(f_k)_{k\geq 1}$ considered as elements in the finite field with $3$ elements. The related formal Laurent series $\sum_{k\geq 0} f_kX^{-k}$ over $\FF_3$ has been identified as a counterexample to the $X$-adic Littlewood conjecture by Adiceam, Nesharim, and Lunnon \cite{ANL} (see also \cite{GR} fur further results). 
A concluding Section~\ref{sec:4} discusses the role of the matrices in this note for different areas of mathematics. 

We will make use of the following notation. For a matrix $C=(c_{i,j})_{i,j\geq 0}$ over any set we denote by $C^{(n,m,k)}$ the $n\times m$ matrix $(c_{i,j})_{0\leq i\leq n, k\leq j\leq k+m}$, in words the upper $n\times m$ submatrix of $C$ starting with column $k$. If $m=n$ we will write $C^{(n,k)}$ instead of $C^{(n,n,k)}$ for the sake of simplicity. If, furthermore, $k=0$ we just write $C^{(n)}$ for $C^{(n,n,0)}$. 

\section{Nice coincidences for the Pascal matrix}\label{sec:2}


There are two ways of defining the so-called Pascal matrix related to Pascal's triangle. The first gives an upper triangular matrix, i.e.,
$$P_1:=\Big(\binom{j}{i}\Big)_{i,j\geq 0},$$ 
where $\binom{j}{i}$ is considered to be zero if $i>j$. 
The second a symmetric matrix, i.e. 
$$P_2:=\Big(\binom{j+i}{i}\Big)_{i,j\geq 0}.$$
As well as the binomial coefficient, the Pascal matrices are well studied objects. 
In the following Theorem A we list some nice and well-known properties of these two versions of the Pascal matrix and add some references for the interested reader. 
\begin{theoremA}
\begin{enumerate}
	\item We have $P_1^TP_1=P_2$.
	(See e.g. \cite[Equ.~(6)]{Lunnon}). 
	\item Let $a\in\ZZ\setminus\{0\}$. Then $P_1^{a}=(\binom{j}{i}a^{j-i})_{i,j\geq 0}=:P_1(a)$. (Cf. e.g. \cite[Proposition~2]{hoflarAA}).
	\item We set $P_1(0)=(\delta_{i,j})_{i,j\geq 0}$, where $\delta_{i,i}=1$ and $\delta_{i,j}=0$ whenever $i\neq j$. Then $P_1(a)P_1(b)=P_1(a+b)$ for all $a,b\in\ZZ$. (This is an easy consequence of Item 2.)
	\item The matrices $P_1(0),\,P_1(1),\,\ldots,\,P_1(p-1)$ all taken $\pmod{p}$ with $p\in\PP$ are qualified to construct a so-called $(0,p)$-sequence in the sense of Niederreiter. These sequences are well-established as Faure sequences in the literature.  (Cf. e.g. \cite[Example~2]{hoflarAA} and the references therein.) 
	\item The matrix $P_2\mod{p}$ with $p\in\PP$ is qualified to construct a normal number in base $p$ with best known discrepancy behavior. (See \cite{hoflarNN1,hoflarNN2,levin99}). 
	\item Every upper left $n\times n$ submatrix of $P_2$ has determinant $1$, i.e. $\det(P_2^{(n)})=1$ for every $n\in\NN$. (See \cite[Theorem 1.2]{bickhogg}.)
	\item Every upper $n\times n$ submatrix of $P_1$ has determinant $1$, i.e., $\det(P_1^{(n,k)})=1$ for every $n\in\NN$ and every $k\in\NN_0$. (See \cite[Theorem~1.1]{bickhogg}.)
	\item Every upper $n\times n$ submatrix of $P_2$ has determinant $1$, i.e., $\det(P_2^{(n,k)})=1$ for every $n\in\NN$ and every $k\in\NN_0$. (This is an easy consequence of Items 7 and 1.)
	
\end{enumerate}
\end{theoremA}
\begin{remark}{\rm
The statement in Item 4 above uses the notion of $(t,s)$-sequences in the sense of Niederreiter. A condition for $s$, $\NN_0\times\NN_0$-matrices $C_1,C_2,\ldots,C_s$ over $\FF_p$, with $p\in\PP$ such that they are qualified to construct a low-discrepancy $(t,s)$-sequence with fixed $t\in\NN_0$ can be formulated as follows. For every integer $m> t$ and all choices of $d_1,d_2,\ldots,d_s\in\NN_0$ such that $d_1+d_2+\cdots+d_s=m-t$ the $(m-t)\times m$ matrix formed by stacking $C_1^{(d_1,m,0)} $, $C_2^{(d_2,m,0)}$, ..., $C_s^{(d_s,m,0)}$ has full row rank $m-t$. (The interested reader is referred to e.g. \cite{DP} for the construction algorithm of the $s$-dimensional sequence based on $C_1,C_2,\ldots,C_s$ over $\FF_p$---the so-called digital method---and for the definition and properties of the \emph{discrepancy} of a sequence.)}
\end{remark}

In the following we define the Pascal matrices over $\ZZ$ with $0$-$1$ entries, which we abbreviate to $M_1$ and $M_2$, by taking the entries of $P_1$ and $P_2$ modulo $2$, i.e. 
$$M_1=\Big(\binom{j}{i}\pmod{2}\Big)_{i,j\geq 0}\quad\mbox{ and }\quad M_2=\Big(\binom{j+i}{i}\pmod{2}\Big)_{i,j\geq 0}.$$
Furthermore, we set for $a\in\ZZ\setminus\{0\}$ 
$$M_1(a):=\Big(\big(\binom{j}{i}\pmod{2}\big) a^{s_2(j)-s_2(i)}\Big)_{i,j\geq 0},$$
where $s_2(n)$ denotes the binary sum-of-digits, i.e. $s_2(n)=n_0+n_1+n_2+\cdots$ where $n=n_0+n_12+n_22^2+\cdots$ is the unique binary representation of $n$ with $n_i\in\{0,1\}$ for all $i\geq 0$. For $a=0$, we set $M(0)=(\delta_{i,j})_{i,j\geq 0}$. \\

Some of the items in Theorem A were already (partially) re-identified for $M_1$ and $M_2$ resp.

	\begin{enumerate}
		\item[(I)] Every upper left $n\times n$ submatrix of $M_2\in\ZZ^{\NN_0\times\NN_0}$ has determinant $ \pm 1$. Indeed we have $\det(M_2^{(n)})=\prod_{k=0}^{n-1}(-1)^{s_2(k)}$ for every $n\in\NN$. (See \cite[Theorem 1.1 (i)]{BaCh}.) This is an analogue to Item 6.
	\item[(II)] Every upper $n\times n$ submatrix of $M_1$ has determinant $ \pm 1$, i.e. $\det(M_1^{(n,k)})=\pm 1$ for every $n\in\NN$ and every $k\in\NN_0$. (See \cite[Theorem 1]{mereb}.) This is an analogue of Item~7 in Theorem~A. 
	\item[(III)] The matrix $M_2$ in $\ZZ^{\NN_0\times\NN_0}$ is qualified to construct a normal number in any base $b\in\NN\setminus\{1\}$ with good discrepancy behavior. (See \cite{levin99}, in which Item (II) was claimed but not proven.) This is an analogue of Item 5 in Theorem A.
	\end{enumerate}
	
In the following Theorem \ref{thm:1} and its corollaries we show some more coincidences between $M_1,\,M_2$ and $P_1,\,P_2$ as well as $M_1(a)$ and $P_1(a)$. \\
	
	We start with the main Theorem \ref{thm:1} that represents an analogue of Item 2 as well as Item 3 in Theorem~A. 
	
	\begin{theorem}\label{thm:1}
	We have $M_1(a)M_1(b)=M_1(a+b)$ for all $a,b\in\ZZ$. 
	\end{theorem}
	
	\begin{proof}We may assume that both $a$ and $b$ are nonzero. Note that the case where $a=0$ or $b=0$ is obvious. 
	We heavily use Lucas' Theorem modulo $2$, which states 
	\begin{equation}\label{eq:LT}
	\binom{j}{i}\pmod{2}=\prod_{k=0}^\infty\binom{j_k}{i_k},
	\end{equation}
	where $i=i_0+i_12+i_22^2+\cdots$ and $j=j_0+j_12+j_22^2+\cdots$ are the binary representations of $i$ and $j$. 
	Obviously, $\binom{j}{i}\mod{2}=1$, if $j_k\geq i_k$ for all $k\geq 0$ and $\binom{j}{i}\mod{2}=0$, else. 
	Here and in the following $[C]_{i,j}$ denotes the coefficient of $C\in\ZZ^{\NN_0\times \NN_0}$ in the $i$th row and $j$th column. 
	
	We compute $[M_1(a)\cdot M_1(b)]_{i,j}$ by using the binary representation of $l=l_0+l_12+l_22^2+\cdots$:
	\begin{align*}
	[M_1(a)\cdot M_1(b)]_{i,j}&=\sum_{l=0}^\infty [M_1(a)]_{i,l}\cdot [M_1(b)]_{l,j}\\
	&=\sum_{l=0}^\infty\Big(\binom{l}{i}\pmod{2}\Big) a^{s_2(l)-s_2(i)}\Big(\binom{j}{l}\pmod{2}\Big) b^{s_2(j)-s_2(l)}\\
	&=a^{-s_2(i)}b^{s_2(j)}\prod_{k=0}^\infty\sum_{l_k=0}^1\binom{j_k}{l_k}\binom{l_k}{i_k}a^{l_k}b^{-l_k}\\
	&=\prod_{k=0}^\infty a^{-i_k}b^{j_k}\sum_{l_k=i_k}^{j_k}{\binom{j_k}{l_k}\binom{l_k}{i_k}a^{l_k}b^{-l_k}}.
	\end{align*}
We observe 
	$$a^{-i_k}b^{j_k}\sum_{l_k=i_k}^{j_k}{\binom{j_k}{l_k}\binom{l_k}{i_k}a^{l_k}b^{-l_k}}=\left\{\begin{array}{ll}0=\binom{j_k}{i_k}(a+b)^{j_k-i_k}& \mbox{if }1=i_k>j_k=0\\ 
	a^{-i_k}b^{j_k}=\binom{j_k}{i_k}(a+b)^{j_k-i_k}& \mbox{if }i_k=j_k=0\\
	a^{-i_k}b^{j_k}\frac{a+b}{b}=\binom{j_k}{i_k}(a+b)^{j_k-i_k}& \mbox{if }0=i_k<j_k=1\\
		a^{-i_k}b^{j_k}\frac{a}{b}=\binom{j_k}{i_k}(a+b)^{j_k-i_k}& \mbox{if }i_k=j_k=1.
	\end{array}\right.$$
	Finally, we arrive at the desired equality,
	\begin{align*}
	[M_1(a)\cdot M_1(b)]_{i,j}&=\prod_{k=0}^\infty a^{-i_k}b^{j_k}\sum_{l_k=i_k}^{j_k}{\binom{j_k}{l_k}\binom{l_k}{i_k}a^{l_k}b^{-l_k}}\\
	&=\prod_{k=0}^\infty\binom{j_k}{i_k}(a+b)^{j_k-i_k}=(\binom{j}{i}\pmod{2})(a+b)^{s_2(j)-s_2(i)}=[M_1(a+b)]_{i,j}.
	\end{align*}
	\end{proof}

	The following Lemma \ref{cor:1} reproves the $LDU$ factorisation of $M_2\in\ZZ^{\NN_0\times \NN_0}$, which was also used in \cite{mereb}. This Lemma \ref{cor:1} is an analogue to Item 1 in Theorem A. 
\begin{lemma}\label{cor:1}
We have $M_1^T\diag\big(((-1)^{t_i})_{i\geq 0}\big) M_1=M_2$ where $t_i:=s_2(i)\pmod{2}$ is the $i$-th element of the Thue--Morse sequence $(t_i)_{i\geq 0}$. 
\end{lemma}
\begin{proof}
We compute by applying Lucas' Theorem modulo $2$ (see equation \eqref{eq:LT}) and by using the binary representation of $l=l_0+l_12+l_22^2+\cdots$:
\begin{align*}
[M_1^T\diag\big((-1)^{t_i})_{i\geq 0}\big) M_1]_{i,j}&=\sum_{l\geq 0}\Big(\binom{i}{l}\pmod{2}\Big)(-1)^{s_2(l)}\Big(\binom{j}{l}\pmod{2}\Big)\\
&=\prod_{k\geq 0}\sum_{l_k=0}^1\binom{i_k}{l_k}(-1)^{l_k} \binom{j_k}{l_k}\\
&=\prod_{k\geq 0}\Big(1+(-1)\binom{i_k}{1}\binom{j_k}{1}\Big)\\
&=\prod_{k\geq 0}\big(\binom{i_k+j_k}{i_k}\pmod{2}\big)\\
&=\binom{i+j}{i}\pmod 2,
\end{align*}
where in the last step we used that $\binom{i+j}{i}\pmod 2=0$ if there occurs at least one carry when adding the binary expansions $i=i_0+i_12+i_22^2+\cdots$ and $j=j_0+j_12+j_22^2+\cdots$, and $\binom{i+j}{i}\pmod 2=1$ else. 

\end{proof}

The following Corollary \ref{cor:2} reproves \cite[Theorem~1.1]{BaCh} that serves as an analogue Item 6 of Theorem A. 
\begin{coro}\label{cor:2} Let $n\in\NN$. Every upper left $n\times n$ submatrix of $M_2$ has determinant $\prod_{i=0}^{n-1}(-1)^{s_2(i)}$, i.e., $\det(M_2^{(n)})=\prod_{i=0}^{n-1}(-1)^{s_2(i)}$ for every $n\in\NN$.\end{coro} 

\begin{proof}
From Lemma \ref{cor:1} we know $(M_1^{(n)})^T\diag((-1)^{t_i})_{0\leq i<n} M_1^{(n)}=M_2^{(n)}$. Since $M_1$ is a non-singular upper triangular matrix with diagonal entries all $1$ we immediately obtain $\det(M_2^{(n)})=\det(\diag((-1)^{t_i})_{0\leq i<n})=\prod_{i=0}^{n-1}(-1)^{s_2(i)}$. 
\end{proof}

The next Corollary \ref{cor:3} even generalizes \cite[Theorem~1]{mereb} the analogue of Item 7 in Theorem A . 
\begin{coro}\label{cor:3}Let $n\in\NN$, $k\in\NN_0$, and $a\neq 0$. The upper $n\times n$ submatrix of $M_1(a)$ starting with column $k$ has determinant $ \pm \prod_{i=0}^{n-1}a^{s_2(i+k)-s_2(i)}$, i.e., $\det(M_1^{(n,k)}(a))=\pm \prod_{i=0}^{n-1}a^{s_2(i+k)-s_2(i)}$ for every $n\in\NN$ and every $k\in\NN_0$. 
\end{coro}

\begin{proof}
We observe first $M_1(a)=\diag((a^{-s_2(i)})_{i\geq 0})M_1(1) \diag((a^{s_2(i)})_{i\geq 0})$. Thus 
$$M_1(a)^{(n,k)}=\diag((a^{-s_2(i)})_{0\leq i<n}) M_1(1)^{(n,k)}\diag((a^{s_2(i)})_{k\leq i<n+k}) .$$
Using \cite[Theorem~1]{mereb} we know $\det(M_1(1)^{(n,k)})=\pm 1$ and the result follows. 
\end{proof}

Finally, Corollary \ref{cor:4} and Remark \ref{rem:1} partially give an analogue of Item 4 of Theorem A, as they ask whether matrices in $\{M_1(a):a\in\NN_0\}$ might be used to construct low-discrepancy $(t,s)$-sequences in the sense of Niederreiter or not. Note that the matrices in $\{P_1(a):a\in\NN_0\}$ are used to construct the low-discrepancy Faure sequences (see Item 4 in Theorem A.)

\begin{coro}\label{cor:4} Let $p\in\PP$ and $a,b\in\ZZ$ such that $a\not\equiv b\pmod{p}$. Then $M_1(a),M_1(b)$ modulo $p$ are qualified to construct a $(0,2)$-sequence over $\FF_p$ in the sense of Niederreiter.\end{coro} 



\begin{proof} Let $c:=b-a\not\equiv 0\pmod{p}$, $m\in\NN$, and $d_1,d_2\geq 0$ such that $d_1+d_2=m$. 
Corollary~\ref{cor:3} ensures that every upper $d_2\times d_2$ submatrix of $M_1(c)$ has determinant $\not\equiv 0 \pmod{p}$. This together with the fact that $M_1(0)=(\delta_{i,j})_{i,j\geq 0}$ immediately yields the linear independence of the system of vectors $([M_1(0)]_{i,0},\ldots,([M_1(0)]_{i,m-1})$, $0\leq i<d_1$ and $([M_1(c)]_{i,0},\ldots,([M_1(c)]_{i,m-1})$, $0\leq i<d_2$. Hence $M_1(0),M_1(c)$ are qualified to construct a $(0,2)$ sequence in base $p$ in the sense of Niederreiter. 
From the fact that $M_1(a)$ is a non-singular upper triangular matrix with determinant $1$ we know then that the system of vectors $([M_1(0)M_1(a)]_{i,0},\ldots,([M_1(0)M_1(a)]_{i,m-1})$, $0\leq i<d_1$ and $([M_1(c)M_1(a)]_{i,0},\ldots,([M_1(c)M_1(a)]_{i,m-1})$, $0\leq i<d_2$ are linear independent modulo $p$. From Theorem~\ref{thm:1} we obtain that $M_1(0)M_1(a)=M_1(a)$, that $M_1(c)M_1(a)=M_1(b)$. Thus $M_1(a),M_1(b)$ are qualified to construct a $(0,2)$-sequence in base $p$ in the sense of Niederreiter. 
\end{proof}

\begin{remark}\label{rem:1}{\rm Theorem A, Item 4 allows to construct three-dimensional low-discrepancy $(0,3)$-sequences over $\FF_p$ with any odd prime $p$ by using, e.g., $P_1(0),P_1(1),P_1(2)$ taken modulo $p\geq 3$. It is not possible to further generalize the statement in Corollary \ref{cor:4} to dimensions $>2$ in the sense of Item 4 in Theorem A. Choose e.g. $C_1=M(0),\,C_2=M(1),\,C_3=M(2)$, $m=3$, and $d_1=d_2=d_3=1$. Then the three vectors $(1,0,0)$, $(1,1,1)$, and $(1,2,2)$ are linearly dependent, which contradicts the condition for a $(0,3)$-sequence in base $p\geq 3$. \\
An interesting question, however, is: does there exist a non-singular upper triangular matrix $C\in\ZZ^{\NN_0\times\NN_0}$ such that $M(0),M(1),C$ are qualified to construct a $(0,3)$-sequence (or at least a $(t,3)$-sequence in base $3$ with small $t\in\NN_0)$ and if so finding an explicit formula for the entries of $C$ might be of interest. Note that the condition on $C_1,C_2,C_3$ that are qualified to construct a $(0,3)$-sequence over $\FF_3$ is quite strict. (See e.g. \cite{hofkos} and \cite{hoflarAA} for discussions on $C_1,C_2$ over $\FF_2$ to be qualified to construct a $(0,2)$-sequence.)}
\end{remark}

\section{The nice coincidences for the Hankel matrix determined by the Catalan numbers}\label{sec:3}

A Hankel matrix $H=(h_{i,j})_{i,j\geq 0}$ can be uniquely described by a sequence $(c_{k})_{k\geq 0}$ as $h_{i,j}=c_{i+j}$ for all $i,j\geq 0$. In this section we focus on Hankel matrices $H_1,H_2$ with $(c_k)_{k\geq 0}$ stemming from Catalan numbers interspersed with zeros, i.e., 
$$c_{2k}=(-1)^kC_k \mbox{ and } c_{2k+1}=0$$
for all $k\geq 0$ with $C_k:=\binom{2k}{k}-\binom{2k}{k-1}=\frac{1}{k+1}\binom{2k}{k}$ once taken pure and once taken modulo ${2}$ but both considered as integers. 

Hence 
$$H_1=(h^{(1)}_{i,j})_{i,j\geq 0}\mbox{ and }H_2=(h^{(2)}_{i,j})_{i,j\geq 0}$$ with $ h^{(1)}_{i,j}=c_{i+j}$ and $h^{(2)}_{i,j}=(c_{i+j}\pmod{2})$. 
Note that as an easy consequence of Lucas' Theorem modulo $2$ (see equation \eqref{eq:LT}), $H_2=(h^{(2)}_{i,j})_{i,j\geq 0}$ with $h^{(2)}_{i,j}=1$ if  $i+j=2^k-2$ for some some $k\in\NN$ and $h^{(2)}_{i,j}=0$ otherwise. \\

We start with a common property of the principal minors of both Hankel matrices $H_1$ and $H_2$.  The $n$th principal minor of a matrix $C$ is the determinant of the upper left $n\times n$ submatrix, i.e. $\det(C^{(n)})$.

\begin{theorem}\label{thm:PM_Hankel}
We have $\det(H_1^{(n)})=\pm 1$ as well as $\det(H_2^{(n)})=\pm 1$ for all $n\in\NN$. 
\end{theorem}
\begin{proof}
The first is a consequence of \cite[Lemma~1]{hofJNT} together with \cite[Proposition~1]{hofJNT}, which ensures $\det(H_1^{(n)})\pmod{p}\not\equiv 0$ for every $n\in\NN$ and every $p\in\PP$. 

The second is an immediate consequence of the ArXiv paper \cite[Theorem~10.1 (i)]{bacher}, which gives the $LDU$ decomposition of $H_2$. (Details on the $LDU$ decomposition can be found in the subsequent Remark~\ref{rem:3}). As \cite{bacher} is an ArXiv paper we give a proof of the second statement of this Theorem~\ref{thm:PM_Hankel}. We observe that $H_2^{(2^k-1)}$ is an anti-diagonal upper triangular matrix with anti-diagonal-entries all equal to $1$ for all $k\in\NN$. Hence, obviously $\det(H_2^{(2^k-1)})=\pm 1$ for all $k$. For $\det(H_2^{(n)})=\pm 1$ for all $2^k-1<n<2^{k+1}-1$ we proceed by induction on $k$. Let $k$ in $\NN$ and assume that $n=2^k-1+l$ with $1\leq l<2^k$. The induction hypothesis ensures $\det(H_2^{(m)})=\pm 1$ for all $m\leq 2^k-1$. We do row as well as column manipulation in $H_2^{(2^k-1+l)}$ without changing the determinant. In more details we exploit the anti-diagonal part denoted by $J_{2l+1}$ in the right bottom of this matrix. We obtain then a matrix $Q$ of the form $Q=\begin{pmatrix}H_2^{(2^k-1-l+1)}&\boldsymbol{0}\\\boldsymbol{0}&J_{2l-1}\end{pmatrix}$ where $\boldsymbol{0}$ stands for a submatrix with all entries equal to $0$. Now obviously, $\det(Q)=\det(H_2^{(2^k-1-l+1)})\cdot \det(J_{2l-1})=\pm 1$, where for the last equality we used the induction hypothesis for $n=2^k-l$. A sketch of the matrices $H_2^{(2^k-1+l)}$ and $Q$, which contains $H_2^{(2^k-1-l+1)}$ and $J_{2l-1}$, can be found in Figure \ref{fig:1}. 

\begin{figure}[h]
	\centering
		\includegraphics[width=0.25\textwidth]{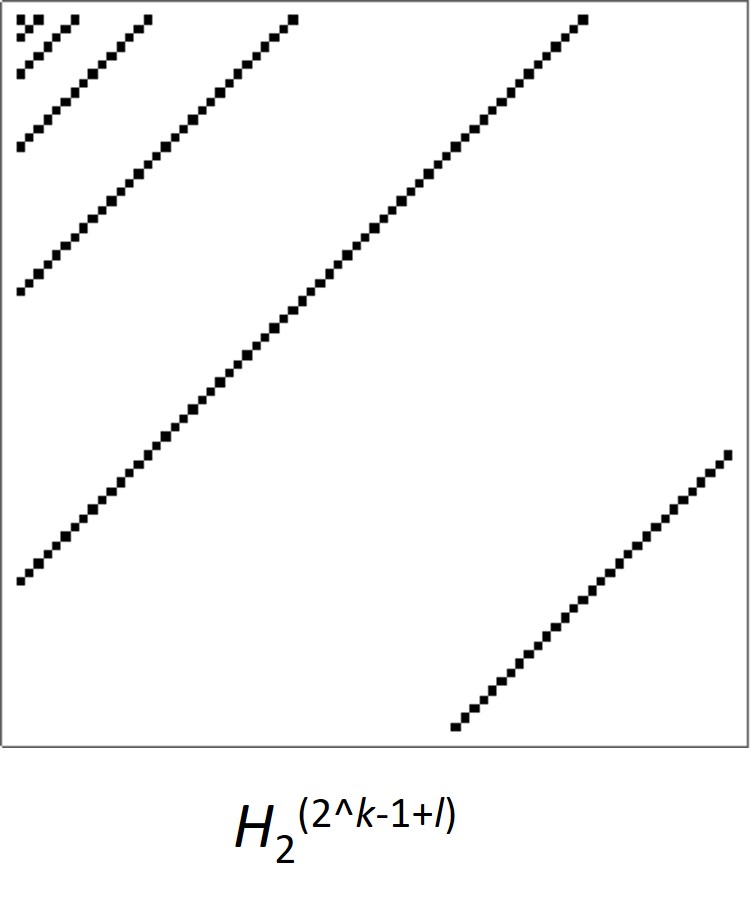}\hspace{1cm}\includegraphics[width=0.50\textwidth]{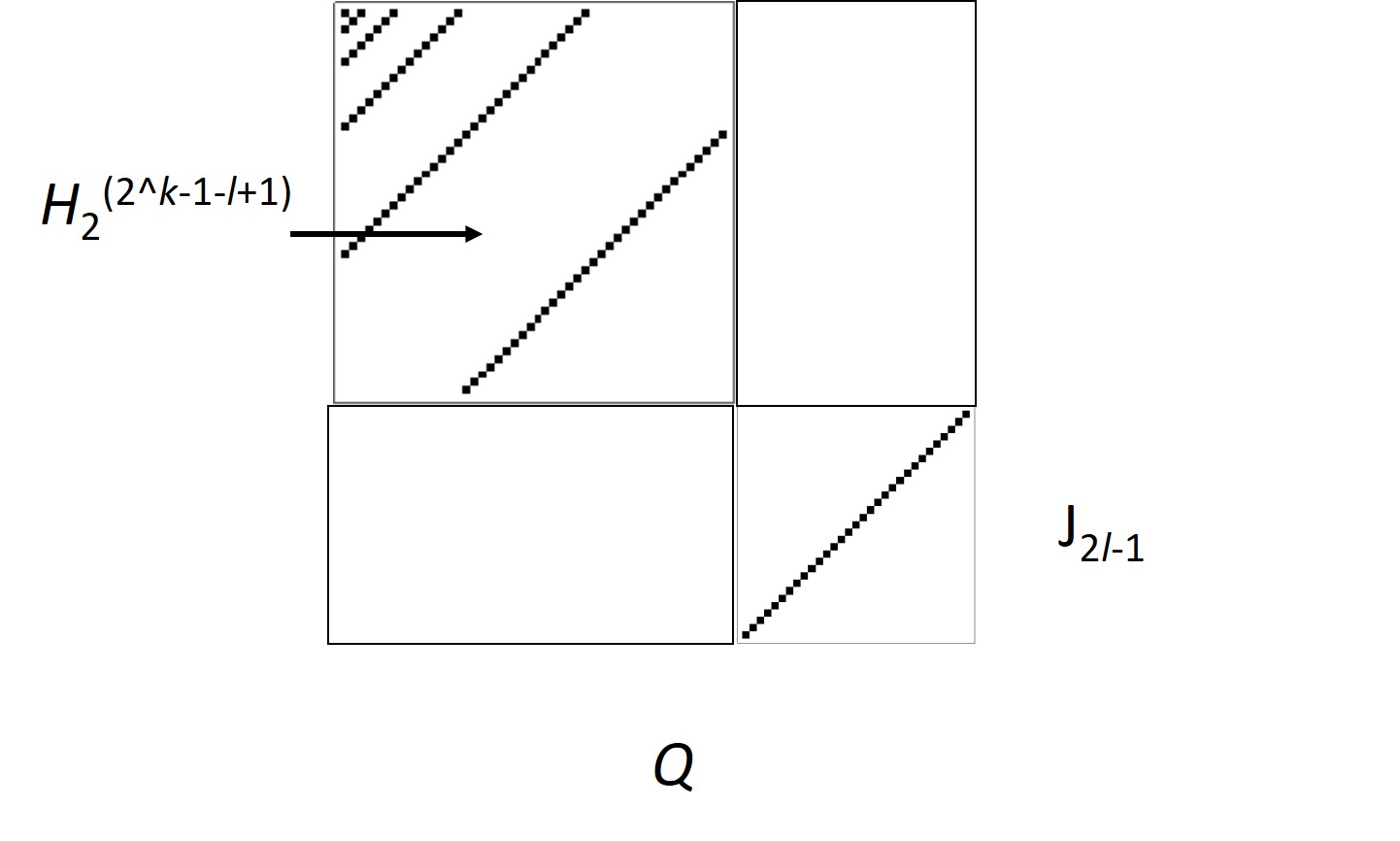}
\caption{Sketch of $H_2^{(2^k-1+l)}$ and $Q$.}\label{fig:1}
\end{figure}

\end{proof}





\begin{remark}\label{rem:3}{\rm 
One consequence of Theorem \ref{thm:PM_Hankel} is that both matrices $H_1$ and $H_2$ have a unique $LDU$ decomposition where $L$ is a lower triangular matrix with diagonal-entries all $1$, $U$ is an upper triangular matrix with diagonal-entries all $1$ and $D$ is a diagonal-matrix with entries $\pm 1$. \cite[Theorem~10.1 (i)]{bacher} gives the $LDU$ decomposition of $H_2$, which was already used as an alternative argument in the proof of Theorem \ref{thm:PM_Hankel}. 
Such $LDU$ decomposition with additionally $U=L^T$ were already investigated e.g. in \cite{PeartWoan}, where $L$ is constructed row wise using a tritriangular so-called Stieltjes matrix. See also \cite{bacher} and \cite{Poo}. In \cite{bacher} the LDU decomposition of some further Hankel matrices are given. One example is the Hankel matrix $H=(h_{i,j})_{i,j\geq 0}$ over $\ZZ$ with $h_{i,j}=C_{i+j}\pmod{2}$ (see \cite[Theorem~10.1]{bacher}). 

A slightly different version of this $LDU$ decomposition is an $LU$ decomposition with $L$ and $U$ be non-singular lower or upper resp. triangular matrices. Theorem~1 and Proposition~1 in \cite{hofJNT} investigate such an $LU$ decomposition of $H_1$ together with the continued fraction expansion of the related formal Laurent series $\sum_{k\geq 0}c_kX^{-k-1}$. The relevant tool for the investigation is a construction of $L$ and $U$ with tritriangular matrices comparable with the Stieltjes matrix in \cite{PeartWoan} but defined via the continued fraction expansions with coefficients having all degrees $1$ of the formal Laurent series. Using \cite[Lemma~1]{hofJNT} together with Theorem~\ref{thm:PM_Hankel} we see that both Laurent series $\sum_{k\geq 0}c_kX^{-k-1}$ as well as $\sum_{k\geq 0}(c_k\pmod 2)X^{-k-1}$ have continued fraction expansion with $[0;A_1(X),A_2(X),A_3(X),\ldots]$ with $\deg(A_i(X))=1$ for all $i\geq 1$. The following proposition gives explicit formulas for those $A_i(X)$. Although these results are not new, we would like to summarise them here in a proposition for the interested reader.}
\end{remark}

\begin{proposition}
Let $\mathcal{L}_1=\sum_{k\geq 0}c_kX^{-k-1}$ and $\mathcal{L}_2=\sum_{k\geq 0}(c_k\pmod 2)X^{-k-1}$ be formal Laurent series over $\QQ$, with $$c_{2k}=(-1)^kC_k \mbox{ and } c_{2k+1}=0.$$
The continued fraction expansions of $\mathcal{L}_1$ and $\mathcal{L}_2$ satisfy $$\mathcal{L}_1=[0;\overline{X}]\quad \mbox{ and } \quad \mathcal{L}_1=[0;{s_1 X},s_2X,s_3X,\ldots]$$
with $(s_i)_{i\geq 1}$ in $\{1,-1\}$ is a paperfolding sequence often denoted by $\mathcal{F}(1,-1,-1,-1,\ldots)$ (cf. e.g. \cite{Alletal} for this notation), that defines $(s_i)_{i\geq 1}=\lim_{i\to\infty}w_i$ via the recursion $w_1=s_1=(1)$, $w_{i+1}=w_i\cdot(-1)\cdot(-{w_i}^R)$ for $i\geq 1$. Here $\cdot$ denotes the concatenation of the finite sequences, $w^R$ denotes the finite sequence $w$ in reflected order, and the $-$ in front of a finite sequence changes the signs of the elements of the sequence. 
\end{proposition}
\begin{proof}
The continued fraction $\mathcal{L}_1=[0;\overline{X}]$ is the starting point in \cite{hofJNT} to obtain the formal Laurent series $\mathcal{L}_1=\sum_{k\geq 0}c_kX^{-k-1}$ determined via the Catalan numbers. Hence the first equality $\mathcal{L}_1=[0;\overline{X}]$ can be found in \cite{hofJNT}. 

The equality $[0,{s_1 X},s_2X,s_3X,\ldots]=\sum_{k\geq 0}(c_k\pmod 2)X^{-k-1}$ can be found in e.g. \cite[Equation~(9)]{Alletal}, previously in e.g. \cite[Theorem~1]{PoSh}. 
\end{proof}

\section{Closing discussions}\label{sec:4}

The content of this note touches on results from several branches of mathematics, as, e.g.:
\begin{itemize}

	\item[-] Problems for a formal Laurent series $\sum_{k\geq j}c_kX^{-k}$ and its simple continued fraction expansions. See e.g. \cite{Alletal} for so-called \emph{folded continued fractions}, that are e.g. continued fractions determined by paperfolding sequences, and their formal Laurant series expansions. (See also e.g.  \cite{bacher}, \cite{hofJNT}, \cite{Poo}, \cite{PoSh}.) For work and discussions on analogs of the Littlewood conjecture in the set of Laurent series see e.g. \cite{ANL}.   
	\item[-] Various aspects on Hankel matrices. As specific $LDU$ or $LU$ decompositions and determinants of submatrices (see e.g. \cite{bacher}, \cite{hofJNT}, \cite{PeartWoan}). 
	\item[-] Properties of the Pascal matrices as e.g. determinants of submatrices (see e.g. \cite{BaCh}, \cite{Lunnon}, \cite{mereb}), their usage for constructing low-discrepancy sequences and normal numbers (see e.g. \cite{faure} and \cite{levin99}).
	\item[-] Finding matrices that are qualified to construct low-discrepancy $(t,s)$-sequences over $\FF_p$ (see e.g. the monograph \cite{DP} for an overview.)	In particular the existence of Hankel matrices over $\FF_p$ that are qualified to construct $(t,s)$-sequences for $s>1$ is an open problem. Existence would contradict the analog of the Littlewood conjecture in the set of formal Laurent series over $\FF_p$. Such $(t,s)$-sequences, which are constructed via Hankel matrices, are well-established as \emph{Kronecker-type sequences}. Kronecker-type sequences are frequently studied (see e.g. \cite{hofFFA} and the references therein). 
	\item[-] Finding normal numbers with low-discrepancy (see \cite{levin99} and also e.g. \cite{BechCart}, \cite{hoflarNN1}, \cite{hoflarNN2}). As $M_2^{(2^n)}$ is a simple column reflection of $M_1^{(2^n)}$ and it might be used for constructing low-discrepancy normal numbers. Also the simple column reflected $M^{(2^n)}_1(a)$ might be a candidate for constructing normal numbers. An open problem for these normal numbers is to identify the exact order of its discrepancies (cf. e.g. \cite{hoflarNN1} and \cite{hoflarNN2}).
	\item[-] Problems concerning apwenian sequences (see e.g. \cite{GHW} and the references therein.). A $0$-$1$ apwenian sequence $(c_j)_{j\geq 0}$ can be defined via determinants of the Hankel matrix $H=(h_{i,j})_{i,j\geq 0}$ with $h_{i,j}=c_{i+j}$, namely $(c_j)_{j\geq 0}$ over $\{0,1\}$ is apwenian if and only if $\det(H^{(n)})\equiv 1\pmod 2$ for every $n\in\NN$. From this definition a connection to the study of determinants of submatrices of Hankel matrices is obvious. Interesting connections between apwenian sequences and the so-called perfect linear complexity profile of pseudo random $0$-$1$ sequences were identified in \cite{Alletal2}. In there a further connection to the related formal Laurent series over $\FF_2$ and its continued fraction expansion is pointed out. 
\end{itemize}

We would like to close this note with the following comment stemming from the interdisciplinary paper \cite{Alletal2} by Allouche, Han, and Niederreiter who motivated the work therein via the sentences ``\textit{We hope that this will help gathering two distinct communities of researchers}'' and ``\textit{One of the intense pleasures in mathematical research is to discover a link between two fields that either did not seem immediately related or were studied from distinct point of views}''. The same two sentences but with ``two'' replaced by ``several'' might be used to summarize the motivation of this note.


\begin{thebibliography}{99}

\bibitem{ANL} Adiceam F., Nesharim E., and Lunnon F., \textit{On the $t$-adic Littlewood conjecture}, Duke Math. J. 170 (2021), no. 10, 2371--2419.

\bibitem{Alletal2} Allouche J.-P., Han G.-N., and Niederreiter H., \textit{Perfect linear complexity profile and apwenian sequences}, Finite Fields Appl. 68 (2020), 101761. 

\bibitem{Alletal} Allouche J.-P., Lubiw A., Mend\'{e}s France M., van der Porten, A.~J., and Shallit J., \textit{Convergents of folded continued fractions}, Acta Arith. 77 (1996), no. 1, 77--96.

\bibitem{bacher} Bacher R., \textit{Paperfolding and Catalan numbers},arXiv paper (2004), \url{https://arxiv.org/abs/math/0406340}. 

\bibitem{BaCh} Bacher R. and Chapman R., \textit{Symmetric Pascal matrices modulo $p$}, European Journal of Combinatorics 25 (2004), 459--473. 


\bibitem{BechCart} Becher V. and Carton O., \textit{Normal numbers and nested perfect necklaces}, J. Complexity 54 (2019), 101403.

\bibitem{bickhogg} Bicknell M. and Hoggart V.E., Jr., \textit{Unit determinants in Pascal triangles}, Fib. Quart. vol 11 (1973), no. 2, 131--144.

\bibitem{faure} Faure H., \textit{Discr\'{e}pance de suites associ\'{e}es \`{a} un syst\`{e}me de num\'{e}ration (en dimension s)}, Acta Arith. 41 (1982), no. 4, 337--351. 

\bibitem{DP} Dick J. and Pillichshammer F., \textit{Digital nets and sequences. Discrepancy theory and quasi-Monte Carlo integration}. Cambridge University Press, Cambridge, 2010.

\bibitem{GR} Garrett S. and Robertson S., \emph{Counterexamples to the $p(t)$-adic Littlewood Conjecture Over Small Finite Fields,} arXiv:2405.14454.

\bibitem{GHW} Guo Y., Han G., and Wu W., \textit{Criteria for apwenian sequences}, Advances in Mathematics 389 (2021), 107899. 

\bibitem{hofFFA} Hofer R., \textit{Kronecker-Halton sequences in
$\FF_p((X^{-1}))$}, Finite Fields Appl. 50 (2018), 154--177. 

\bibitem{hofJNT} Hofer R., \textit{Finding both, the continued fraction and the Laurent series expansion of golden ratio analogs in the field of formal power series}, J. Number Theory 223 (2021), 168--194.

\bibitem{hofkos} Hofer R. and Suzuki, K., \textit{A complete classification of digital $(0,m,3)$
-nets and digital $(0,2)$
-sequences in base $2$}, Unif. Distrib. Theory 14 (2019), No. 1, 43--52.

\bibitem{hoflarAA} Hofer R. and Larcher G., \textit{On existence and discrepancy of certain digital Niederreiter-Halton sequences}, Acta Arith. 141 (2010), no. 4, 369--394.

\bibitem{hoflarNN1} Hofer R. and Larcher G., \textit{The exact order of discrepancy for Levin's normal number in base 2}, J. Th\'{e}or. Nombres Bordeaux 35 (2023), no. 3, 999--1023. 

\bibitem{hoflarNN2} Hofer R. and Larcher G., \textit{Discrepancy bounds for normal numbers generated by necklaces in arbitrary base}, J. Complexity 78 (2023), Paper No. 101767, 26 pp.

\bibitem{levin99} Levin B.~M., \textit{On the discrepancy estimate of normal numbers}, Acta Arithmetica, vol. 88 (1999) no. 2, pp.
99--111.

\bibitem{Lunnon} Lunnon W.~F., \textit{The Pascal matrix}, Fib. Quart. vol. 15 (1977), no. 3, 201--204. 



\bibitem{mereb} Mereb M., \textit{On determinants of matrices related to the Pascal's triangle}, Periodica Mathematica Hungarica 89 (2024), 168--174. 
 

\bibitem{PeartWoan} Peart P. and Woan W.-J., \textit{Generating Functions via Hankel and Stieltjes Matrices}, Journal of Integer Sequences 3 (2000), no. 2, Article 00.2.1. 

\bibitem{Poo} van der Poorten A., \textit{Formal power series and their continued fraction expansion},  Algorithmic Number Theory (1998), pp. 358--371, Proceedings of the Third International Symposium, ANTS-III, Portland, Orgeon, USA, June 21-25, 1998. 

\bibitem{PoSh}  van der Poorten A.~J. and Shallit J., \textit{Folded continued fractions}, J. Number Theory 40 (1992), no. 2, 237--250. 

\end{thebibliography}
\end{document}